\documentclass[11pt]{article}
\usepackage{amssymb}
\usepackage{amsfonts}
\usepackage{amsmath}
\usepackage{amscd}
\usepackage{amsthm}
\usepackage{latexsym}
\usepackage[all]{xy}
\usepackage{CJK}
\usepackage{setspace}
\usepackage{indentfirst}
\usepackage{mathabx}
\usepackage{tikz}
\usetikzlibrary{matrix,arrows,decorations.pathmorphing}

\numberwithin{equation}{section}
\newtheorem{theorem}{Theorem}
\newtheorem{lemma}{Lemma}[section]
\newtheorem{thm}[lemma]{Theorem}
\newtheorem{prop}[lemma]{Proposition}
\newtheorem{definition}[lemma]{Definition}
\newtheorem{corollary}[lemma]{Corollary}
\newtheorem{remark}[lemma]{Remark}
\newtheorem{example}[lemma]{Example}

\newcommand{\add}{\rm{add}}
\newcommand{\pd}{\rm{pd}}
\newcommand{\Ext}{\rm{Ext}}
\newcommand{\End}{\rm{End}}
\newcommand{\Hom}{\rm{Hom}}
\newcommand{\mo}{\rm{mod}}
\newcommand{\id}{\rm{id}}
\newcommand{\dom}{\rm{domdim}}
\newcommand{\gl}{\rm{gld}}

\newcommand{\ind}{\rm{Ind}}
\newcommand{\soc}{\rm{soc}}
\newcommand{\prinj}{\rm{prinj}}
\newcommand{\sub}{\rm{Sub}}
\newcommand{\fac}{\rm{Fac}}

\begin{document}
\title{\bf Higher Auslander algebras of finite representation type }
\author{Shen Li}
\date{}
\maketitle
\vspace{-2em}

\begin{center}\section*{}\end{center}

\begin{abstract}
Let $\Lambda$ be an $n$-Auslander algebra with global dimension $n+1$. In this paper, we prove that $\Lambda$ is representation-finite if and only if the number of non-isomorphic indecomposable $\Lambda$-modules with projective dimension $n+1$ is finite. As an application, we classify the representation-finite higher Auslander algebras of linearly oriented type $\mathbb{A}$ in the sense of Iyama and calculate the number of non-isomorphic indecomposable modules over these algebras.
\end{abstract}

{\small {\bf Key words and phrases:} \ {\rm Higher Auslander algebras, Higher Auslander algebras of linearly oriented type $\mathbb{A}$, Algebras of finite representation type}}

{\small {\bf Mathematics Subject Classification 2010:} \ {\rm 16E10, 16G10, 16G60}}

\vskip0.2in
\section{Introduction}
Let $A$ be a representation-finite algebra and $M_{1},M_{2},\cdots,M_{m}$ be a complete set of non-isomorphic indecomposable $A$-modules. Then $M=\oplus^{m}_{i=1}M_{i}$ is an additive generator of the category $A$-{\mo} and the endomorphism algebra $\Lambda={\End}_{A}M$ is called the \emph{Auslander algebra} of $A$. Auslander algebras are characterized as algebras with global dimension at most 2 and dominant dimension at least 2, see \cite{A}. Recall that the \emph{dominant dimension of an algebra $\Gamma$} is the largest integer $d\geq 0$ such that the terms $I_{0}, I_{1}, \cdots,I_{d-1}$ in the minimal injective resolution of $\Gamma$ are projective. An interesting question is to consider whether the Auslander algebra $\Lambda$ of a representation-finite algebra $A$ is still of finite representation type. Auslander and Reiten proved in \cite{AR} that $\Lambda$ is representation-finite if and only if the triangular matrix algebra $T_{2}(A)= \left( \begin{smallmatrix}
                       A & 0 \\
                       A & A
                      \end{smallmatrix}
                      \right) $
is representation-finite. A classification of representation-finite Auslander algebras $\Lambda$ in terms of the universal cover of the Auslander-Reiten quiver of $A$ was obtained in \cite{IPTZ}.

Higher Auslander algebras were introduced by Iyama in \cite{I1} which are a generalisation of Auslander algebras. Recall that an algebra $\Lambda$ is called an \emph{$n$-Auslander algebra} if its global dimension is at most $n+1$ and dominant dimension is at least $n+1$. Any $n$-Auslander algebra arises as the endomorphism algebra of an $n$-cluster tilting module over some algebra, see \cite{I1,I2}. Both $n$-Auslander algebras and $n$-cluster tilting modules are central objects of study in higher homological algebra. For the case $n=1$, they are precisely classical Auslander algebras and additive generators for representation-finite algebras, respectively. Then it is natural to ask when an $n$-Auslander algebra is of finite representation type.

Note that the largest possible projective dimension of a module over an $n$-Auslander algebra is $n+1$. Moreover, a module has projective dimension $n+1$ if and only if its socle has projective dimension $n+1$, see \cite{LMZ}. In this paper, we show that every indecomposable non-projective module with projective dimension at most $n$ occurs as a syzygy in the projective resolution of those with projective dimension $n+1$ and prove the following result.

\begin{theorem}{\rm (Theorem 3.2)} Let $\Lambda$ be an $n$-Auslander algebra. Then $\Lambda$ is representation-finite if and only if the number of non-isomorphic indecomposable $\Lambda$-modules with projective dimension $n+1$ is finite.
\end{theorem}

Higher Auslander algebras of linearly oriented type $\mathbb{A}$ introduced in \cite{I3} are a remarkable family of algebras with finite global dimension admitting higher cluster tilting modules. Following a recursive construction, this family of algebras has a rich combinatorial structure and gains a lot of attention in representation theory, see for example \cite{DJW,HJ,JK,OT}. As an application of Theorem 1, we classify the representation-finite higher Auslander algebras of linearly oriented type $\mathbb{A}$ and calculate the number of non-isomorphic indecomposable modules.

\begin{theorem}{\rm(Theorem 4.8)}Let $A_{m}^{n}$ be the $(n-1)$-Auslander algebra of linearly oriented type $A_{m}$$(n,m\geq2)$ and $|{\ind}\,A_{m}^{n}|$ denote the number of non-isomorphic indecomposable $A_{m}^{n}$-modules. Then $A_{m}^{n}$ is representation-finite if and only if one of the followings holds:

 {\rm(1)} $m=2$. In this case, $|{\ind}\,A_{2}^{n}|=2n+1$.

 {\rm(2)} $m=3$. In this case, $|{\ind}\,A_{3}^{n}|=\frac{(n+1)(n^{2}+5n+3)}{3}$.

 {\rm(3)} $n=2$ and $m=4$. In this case, $|{\ind}\,A_{4}^{2}|=56$.
\end{theorem}

This paper is organised as follows. In section 2, we recall some results on the projective dimension of modules over $n$-Auslander algebras.  In section 3, we provide a necessary and sufficient condition for higher Auslander algebras to be representation-finite. As an application, we classify representation-finite higher Auslander algebras of linearly oriented type $\mathbb{A}$ in section 4.

\section{Preliminaries}
Throughout this paper, all algebras are finite dimensional algebras over an algebraically closed field $k$ and all modules are finitely generated left modules. For an algebra $A$, we denote by $A$-{\mo} the category of finitely generated $A$-modules and by $D$ the standard duality functor ${\Hom}_{k}(-,k)$. The global dimension and dominant dimension of $A$ are denoted by ${\gl}\,A$ and ${\dom}\,A$, respectively. All subcategories of $A$-mod are assumed to be full. For a subcategory $\mathcal{T}$ of $A$-mod, we write $$\mathcal{T}^{\perp}=\{X\in A\text{-}{\mo}\,|\,{\Hom}_{A}(\mathcal{T},X)=0\},$$$$^{\perp}\mathcal{T}=\{X\in A\text{-}{\mo}\,|\,{\Hom}_{A}(X,\mathcal{T})=0\}.$$ For an $A$-module $M$, ${\soc}\,M$ stands for the socle of $M$ and ${\pd}\,M$ stands for the projective dimension of $M$. The subcategory ${\add}\,M$ of $A$-mod consists of direct summands of finite direct sums of $M$. We denote by ${\fac}\,M$ and ${\sub}\,M$ the subcategories consisting of $A$-modules generated and cogenerated by $M$, respectively.

The aim of this section is to show basic results on projective dimension of modules over higher Auslander algebras. These will be useful in the proof of our main results.

Let $A$ be an algebra and $n$ be a positive integer. An $A$-module $M$ is called an \emph{$n$-cluster tilting module} if it satisfies

\qquad ${\add}\,M=\{X\in A\text{-}{\mo}\,| \,{\Ext}^{i}_{A}(X,M)=0\;{\rm for} \;0<i<n\}$

\qquad \qquad \quad $=\{X\in A\text{-}{\mo}\,|\,{\Ext}^{i}_{A}(M,X)=0\;{\rm for}\; 0<i<n\}.$

Let $\Lambda$ be an $n$-Auslander algebra. According to the higher Auslander correspondence in \cite{I1}, there exists a basic $n$-cluster tilting module $M$ over an algebra $A$ such that $\Lambda\cong{\End}_{A}M$. The following result gives the projective dimensions of all simple $\Lambda$-modules. It is similar to \cite[Theorem 3.1]{LMZ} and we contain a proof here.

\begin{prop}
Let $M=\oplus^{t}_{i=1}M_{i}$ with $M_{i}$ indecomposable and $S_{i}$ be the top of the indecomposable projective $\Lambda$-module ${\Hom}_{A}(M_{i},M)$. Then we have

{\rm(1)} If $M_{i}$ is an injective $A$-module, the projective dimension of $S_{i}$ is at most $n$.

{\rm(2)} If $M_{i}$ is not injective, the projective dimension of $S_{i}$ is $n+1$.
\end{prop}

\begin{proof}
(1) Assume $M_{i}$ is an injective $A$-module. Then the canonical epimorphism $M_{i}\rightarrow M_{i}/{\soc}\,M_{i}$ is a left almost split morphism.
According to \cite[Lemma 6.9(b), IV]{ASS}, applying ${\Hom}_{A}(-,M)$ gives rise to the following short exact sequence

\qquad $0\rightarrow{\Hom}_{A}(M_{i}/{\soc}\,M_{i},M)\rightarrow {\Hom}_{A}(M_{i},M)\rightarrow S_{i}\rightarrow 0.\quad (\ast)$

Now consider $M_{i}/{\soc}\,M_{i}$. By \cite[Theorem 2.2.3]{I2}, there exists a long exact sequence $0\rightarrow M_{i}/{\soc}\,M_{i}\rightarrow M^{1}\rightarrow \cdots \rightarrow M^{n}\rightarrow 0 $ with $M^{j}\in {\add}\,M$ for $1\leq j\leq n$ such that the following sequence is exact.
$$0\rightarrow {\Hom}_{A}(M^{n},M)\rightarrow \cdots \rightarrow {\Hom}_{A}(M^{1},M)\rightarrow {\Hom}_{A}(M_{i}/{\soc}\,M_{i},M)\rightarrow0 (\ast\ast)$$
Combining $(\ast)$ and $(\ast\ast)$, we get the following projective resolution of $S_{i}$.
$$0\rightarrow {\Hom}_{A}(M^{n},M)\rightarrow \cdots \rightarrow {\Hom}_{A}(M^{1},M)\rightarrow {\Hom}_{A}(M_{i},M)\rightarrow S_{i}\rightarrow0 $$
This implies that $S_{i}$ has projective dimension at most $n$.

(2) Assume $M_{i}$ is not injective. According to \cite[Theorem 3.3.1]{I2}, there exists an $n$-almost split sequence $0\rightarrow M_{i}\rightarrow N^{1}\rightarrow \cdots \rightarrow N^{n}\rightarrow N^{n+1}\rightarrow 0$ with $N^{j}\in {\add}\,M$ for $1\leq j\leq n+1$. Applying ${\Hom}_{A}(-,M)$ yields the following minimal projective resolution of $S_{i}$.
$$0\rightarrow {\Hom}_{A}(N^{n+1},M)\rightarrow \cdots \rightarrow {\Hom}_{A}(N^{1},M)\rightarrow {\Hom}_{A}(M_{i},M)\rightarrow S_{i}\rightarrow 0$$
Thus $S_{i}$ has projective dimension $n+1$.
\end{proof}

We recall two characterizations of higher Auslander algebras in terms of the projective dimension of modules.

Denote by ${\prinj}(\Lambda)$ the subcategory consisting of all projective-injective $\Lambda$-modules.

\begin{prop}{\rm \cite[Corollary 4.2]{LMZ}}
Let $\Lambda$ be an algebra with global dimension $n+1$. Then $\Lambda$ is an $n$-Auslander algebra if and only if it satisfies

\qquad \quad ${\prinj}(\Lambda)=\{I\in\Lambda\text{-}{\mo}\,|\, I\ \text{is injective and} \ {\pd(\soc}\,I)\leq n\}.$
\end{prop}

Denote by $P^{\leq n}(\Lambda)$ the subcategory consisting of all $\Lambda$-modules with projective dimension at most $n$.
\begin{prop}{\rm \cite[Corollary 4.4]{LMZ}}
Let $\Lambda$ be an algebra. Then $\Lambda$ is an $n$-Auslander algebra if and only if it satisfies

\qquad \qquad $P^{\leq n}(\Lambda)=\{X\in \Lambda\text{-}{\mo}\,|\,{\pd(\soc}\,X)\leq n\}={\sub}\,\Lambda.$
\end{prop}

The above proposition shows that a module over an $n$-Auslander algebra has projective dimension $n+1$ if and only if its socle has projective dimension $n+1$.

\begin{definition} \rm
A pair $(\mathcal{T},\mathcal{F})$ of subcategories of $\Lambda$-{\mo} is called a \emph{torsion pair} if the following conditions are satisfied:

(1) ${\Hom}_{\Lambda}(X,Y)=0$ for all $X\in\mathcal{T}, \,Y\in{\mathcal{F}}$.

(2) ${\Hom}_{\Lambda}(X,\mathcal{F})=0$ implies $X\in \mathcal{T}$.

(3) ${\Hom}_{\Lambda}(\mathcal{T},Y)=0$ implies $Y\in \mathcal{F}$.
\end{definition}

If $(\mathcal{T},\mathcal{F})$ is a torsion pair, $\mathcal{T}$ is called a \emph{torsion class} and $\mathcal{F}$ is called a \emph{torsionfree class}. Moreover, we have $\mathcal{T}^{\perp}=\mathcal{F}$ and $^{\perp}\mathcal{F}=\mathcal{T}$.

\begin{definition} \rm
A triple $(\mathcal{C},\mathcal{T},\mathcal{F})$ of subcategories of $\Lambda$-{\mo} is called a \emph{torsion-torsionfree triple} (or \emph{TTF-triple} for short) if both $(\mathcal{C},\mathcal{T})$ and $(\mathcal{T},\mathcal{F})$ are torsion pairs in $\Lambda$-{\mo}.
\end{definition}

If $(\mathcal{C},\mathcal{T},\mathcal{F})$ is a TTF-triple, $\mathcal{T}$ is called a \emph{TTF-class}.

Next we show the existence of the TTF-triple in the category of modules over higher Auslander algebras.

For an $n$-Auslander algebra $\Lambda$, there exists a basic projective-injective $\Lambda$-module $Q$ such that ${\add}\,Q={\prinj}(\Lambda)$. Since the dominant dimension ${\dom}\,\Lambda\geq1$, any indecomposable projective $\Lambda$-module can be embedded into $Q$ and by Proposition 2.3, we have $${\sub}\,Q={\sub}\,\Lambda=P^{\leq n}(\Lambda).$$

Denote by $\upsilon$ the Nakayama functor of $\Lambda$-{\mo}. Due to the isomorphism ${\Hom}_{\Lambda}(-,Q)\cong D{\Hom}_{\Lambda}(\upsilon^{-1}Q,-)$, we have $^{\perp}Q=(\upsilon^{-1}Q)^{\perp}$. Since $Q$ is projective-injective such that ${\add}\,Q={\prinj}(\Lambda)$, by Proposition 2.2, the socle of $Q$ is exactly the direct sum of all simple modules with projective dimension at most $n$. For a $\Lambda$-module $X$, if ${\Hom}_{\Lambda}(X,Q)=0$, it has no composition factors with projective dimension at most $n$. Thus we have $$^{\perp}Q=\{X\in \Lambda\text{-}{\mo}\,|\, {\pd}\,S=n+1 \;\text{for any composition factor} \;S\; \text{of}\; X\}.$$

In particular, any simple module with projective dimension $n+1$ belongs to $^{\perp}Q$. For a $\Lambda$-module $X$ with projective dimension $n+1$, it either belongs to  $^{\perp}Q$ or has a composition factor with projective dimension at most $n$.

\begin{prop}
Let $\Lambda$ be an $n$-Auslander algebra and $Q$ be a basic projective-injective module such that ${\add}\,Q={\prinj}(\Lambda)$. Then the following statements hold:

{\rm(1)} $({\fac}\,(\upsilon^{-1}Q),\,^{\perp}Q,\, {\sub}\,Q)$ is a TTF-triple in $\Lambda$-${\mo}$.

{\rm(2)} Let $M$ be an $n$-cluster tilting module over an algebra $A$ such that $\Lambda\cong{\End}_{A}M$. Then there exists an equivalence between ${\fac}\,(\upsilon^{-1}Q)\cap{\sub}\,Q$ and $A\text{-}{\mo}$.

{\rm(3)} The global dimension ${\gl}\,A\leq n$ if and only if ${\fac}\,(\upsilon^{-1}Q)\subseteq{\sub}\,Q$.
\end{prop}

\begin{proof}
{\rm(1)} The dominant dimension ${\dom}\,\Lambda\geq1$ implies that the injective envelope of $\Lambda$ is contained in ${\add}\,Q$. According to \cite[Theorem 3.1]{J}, we know that $^{\perp}Q=\{X\in\Lambda\text{-}{\mo}\,|\,{\Hom}_{\Lambda}(X,Q)=0\}$ is a TTF-class. We only need to prove ${\sub}\,Q=(^{\perp}Q)^{\perp}$ and ${\fac}\,(\upsilon^{-1}Q)=$$^{\perp}(^{\perp}Q)$.

If $X\in(^{\perp}Q)^{\perp}$, we have ${\Hom}_{\Lambda}(S,X)=0$ for any simple module $S$ with projective dimension $n+1$. Then the socle of $X$ has projective dimension at most $n$. By Proposition 2.3, $X\in {\sub}\,Q$ and this shows $(^{\perp}Q)^{\perp}\subseteq{\sub}\,Q$. Assume $Y\in{\sub}\,Q$. There exists a monomorphism $0\rightarrow Y \rightarrow Q^{d}$ for some integer $d\geq0$. Then we have $0\rightarrow {\Hom}_{\Lambda}(^{\perp}Q,Y)\rightarrow {\Hom}_{\Lambda}(^{\perp}Q,Q^{d})=0$ and $Y\in(^{\perp}Q)^{\perp}$ which means ${\sub}\,Q\subseteq (^{\perp}Q)^{\perp}$. Thus ${\sub}\,Q=(^{\perp}Q)^{\perp}$ and  $(^{\perp}Q,\, {\sub}\,Q)$ is a torsion pair.

If $X\in$$^{\perp}(^{\perp}Q)$, we have ${\Hom}_{\Lambda}(X,S)=0$ for any simple module $S$ with projective dimension $n+1$. Then the top of $X$ has projective dimension at most $n$. By Proposition 2.2, ${\add}(\upsilon^{-1}Q)$ consists of all the projective modules whose tops have projective dimension at most $n$. Thus $X$ is generated by $\upsilon^{-1}Q$ and $^{\perp}(^{\perp}Q)\subseteq{\fac}\,(\upsilon^{-1}Q)$. Assume $Y\in{\fac}\,(\upsilon^{-1}Q)$. There exists an epimorphism $(\upsilon^{-1}Q)^{d}\rightarrow Y\rightarrow 0$ for some integer $d\geq0$. Then we have $0\rightarrow {\Hom}_{\Lambda}(Y,(\upsilon^{-1}Q)^{\perp})\rightarrow {\Hom}_{\Lambda}((\upsilon^{-1}Q)^{d},(\upsilon^{-1}Q)^{\perp})=0$ and ${\Hom}_{\Lambda}(Y,(\upsilon^{-1}Q)^{\perp})=0$ which means ${\fac}\,(\upsilon^{-1}Q)\subseteq $$^{\perp}((\upsilon^{-1}Q)^{\perp})=^{\perp}$$(^{\perp}Q)$. Thus ${\fac}\,(\upsilon^{-1}Q)=$$^{\perp}(^{\perp}Q)$ and $({\fac}\,(\upsilon^{-1}Q),\,^{\perp}Q)$ is a torsion class.

{\rm(2)} By Proposition 2.2, the top of $\upsilon^{-1}Q$ is exactly the direct sum of all simple modules with projective dimension at most $n$. Then according to Proposition 2.1, we have $\upsilon^{-1}Q\cong{\Hom}_{A}(DA,M)$. Since $M$ is a cogenerator, it follows from the dual of \cite[Lemma 4.2]{KZ} that ${\Hom}_{A}(-,M)$ is fully faithful. Then we have ${\End}_{\Lambda}\upsilon^{-1}Q\cong{\End}_{\Lambda}{\Hom}_{A}(DA,M) \cong{\Hom}_{A}(DA,DA)\cong A.$ Note that $\upsilon^{-1}Q$ is a projective generator in ${\fac}\,(\upsilon^{-1}Q)$. According to \cite[Corollary 3.3]{G}, the subcategory ${\fac}\,(\upsilon^{-1}Q)\cap{\sub}\,Q$ is equivalent to the category $A$-{\mo} of $A$-modules.

{\rm(3)} For the injective $\Lambda$-module $D\Lambda$, there is a decomposition $D\Lambda\cong Q\oplus I$ where the injective module $I$ has no non-zero projective direct summands. By Proposition 2.2, the socle of $I$ is exactly the direct sum of all the simple modules with projective dimension $n+1$. It follows from \cite[Theorem 1.20]{I3} that ${\gl}\,A\leq n$ if and only if ${\Ext}_{\Lambda}^{i}(I,\Lambda)=0$ for any $1\leq i \leq n$.

Assume ${\Ext}_{\Lambda}^{i}(I,\Lambda)=0$ for any $1\leq i \leq n$. Then we have ${\Hom}_{\Lambda}(I,\Lambda)=0$ by \cite[Lemma 2.3(b)]{I3}. So $I\in\!^{\perp}\Lambda\subseteq\!^{\perp}Q=(\upsilon^{-1}Q)^{\perp}$ and ${\Hom}_{\Lambda}(\upsilon^{-1}Q,I)=0$. This implies that $\upsilon^{-1}Q$ has no composition factors with projective dimension $n+1$. Thus any module generated by $\upsilon^{-1}Q$ has projective dimension at most $n$. It follows that ${\fac}\,(\upsilon^{-1}Q)\subseteq P^{\leq n}(\Lambda)={\sub}\,Q$.

Now assume ${\fac}\,(\upsilon^{-1}Q)\subseteq{\sub}\,Q$. Since the torsionfree class ${\sub}\,Q=P^{\leq n}(\Lambda)$ is closed under submodules, any submodule of the modules in ${\fac}\,(\upsilon^{-1}Q)$ has projective dimension at most $n$. Thus $\upsilon^{-1}Q$ has no composition factors with projective dimension $n+1$. Then ${\Hom}_{\Lambda}(\upsilon^{-1}Q,I)=0$ and $I\in(\upsilon^{-1}Q)^{\perp}=\!^{\perp}Q$. Since $\Lambda$ is an $n$-Auslander algebra, according to \cite[Proposition 1.4]{I4}, any simple module $S$ with projective dimension $n+1$ satisfies ${\Ext}^{i}_{\Lambda}(S,\Lambda)=0$ for $1\leq i \leq n$. Note that $^{\perp}\,Q$ consists of modules all of whose composition factors have projective dimension $n+1$. Then $I\in\!^{\perp}Q$ satisfies ${\Ext}^{i}_{\Lambda}(I,\Lambda)=0$ for $1\leq i \leq n$. As a result, we have ${\gl}\,A\leq n$.
\end{proof}

\begin{remark}\rm
{\rm (1)} Proposition 2.6(2) shows that if $A$ is representation-infinite, so is $\Lambda$.

{\rm(2)} A module is called to be \emph{torsionless} if it can be cogenerated by a projective module. An algebra is called to be \emph{torsionless-finite} if there are finitely many non-isomorphic indecomposable torsionless modules. By Proposition 2.6(2), if $\Lambda$ is torsionless-finite, $A$ is representation-finite.

{\rm(3)} If ${\gl}\,A\leq n$, there exists an equivalence ${\fac}\,(\upsilon^{-1}Q)\cong A\text{-}{\mo}$.
\end{remark}

We give an example to show that the inclusion ${\fac}\,(\upsilon^{-1}Q)\subseteq{\sub}\,Q$ does not always hold.

\begin{example} \rm
Let $\Lambda$ be an algebra given by the following quiver with relations $\alpha_{1}\alpha_{2}=\alpha_{2}\alpha_{3}=\alpha_{4}\alpha_{5}=\alpha_{5}\alpha_{6}=0$.

\qquad\qquad \qquad \qquad $\small{1\xleftarrow{\alpha_{1}}2\xleftarrow{\alpha_{2}}3\xleftarrow{\alpha_{3}}4\xleftarrow{\alpha_{4}}5\xleftarrow{\alpha_{5}}6\xleftarrow{\alpha_{6}}7}$

\noindent Then $\Lambda$ is a 2-Auslander algebra and $Q={2\atop 1}\oplus{3\atop 2}\oplus\begin{smallmatrix} 5\\4\\3 \end{smallmatrix}\oplus{6\atop 5}\oplus{7\atop 6}$. The Auslander-Reiten quiver of $\Lambda$ is as follows.

\noindent $
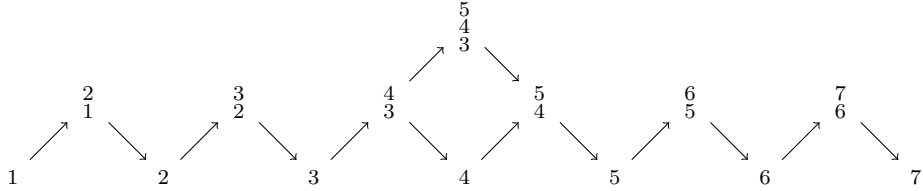
\begin{tikzpicture}[->][scale=.85]
\node(s1) at (0,0) {$\begin{smallmatrix} 1 \end{smallmatrix}$};
\node(s2) at (2,0) {$\begin{smallmatrix} 2 \end{smallmatrix}$};
\node(s3) at (4,0) {$\begin{smallmatrix} 3 \end{smallmatrix}$};
\node(s4) at (6,0) {$\begin{smallmatrix} 4 \end{smallmatrix}$};
\node(s5) at (8,0) {$\begin{smallmatrix} 5 \end{smallmatrix}$};
\node(s6) at (10,0){$\begin{smallmatrix} 6 \end{smallmatrix}$};
\node(s7) at (12,0){$\begin{smallmatrix} 7 \end{smallmatrix}$};
\node(t1) at (1,1) {$\begin{smallmatrix} 2\\1 \end{smallmatrix}$};
\node(t2) at (3,1) {$\begin{smallmatrix} 3\\2 \end{smallmatrix}$};
\node(t3) at (5,1) {$\begin{smallmatrix} 4\\3 \end{smallmatrix}$};
\node(t4) at (7,1) {$\begin{smallmatrix} 5\\4 \end{smallmatrix}$};
\node(t5) at (9,1) {$\begin{smallmatrix} 6\\5 \end{smallmatrix}$};
\node(t6) at (11,1){$\begin{smallmatrix} 7\\6 \end{smallmatrix}$};
\node(r1) at (6,2) {$\begin{smallmatrix} 5\\4\\3 \end{smallmatrix}$};

\draw (s1)--(t1);
\draw (t1)--(s2);
\draw (s2)--(t2);
\draw (t2)--(s3);
\draw (s3)--(t3);
\draw (t3)--(s4);
\draw (t3)--(r1);
\draw (r1)--(t4);
\draw (s4)--(t4);
\draw (t4)--(s5);
\draw (s5)--(t5);
\draw (t5)--(s6);
\draw (s6)--(t6);
\draw (t6)--(s7);
\end{tikzpicture}
$

\noindent We have ${\fac}\,(\upsilon^{-1}Q)={\add}\,(\begin{smallmatrix} 1\end{smallmatrix}\oplus \begin{smallmatrix} 2\end{smallmatrix}\oplus \begin{smallmatrix}3 \end{smallmatrix}\oplus \begin{smallmatrix} 5 \end{smallmatrix}\oplus \begin{smallmatrix} 6\end{smallmatrix}\oplus{2\atop 1}\oplus{3\atop 2}\oplus{5\atop 4}\oplus{6\atop 5}\oplus\begin{smallmatrix} 5\\4\\3 \end{smallmatrix})$,

\qquad $^{\perp}Q={\add}\,(\begin{smallmatrix} 4\end{smallmatrix}\oplus \begin{smallmatrix} 7\end{smallmatrix})$ and

\qquad ${\sub}\,Q={\add}\,(\begin{smallmatrix} 1\end{smallmatrix}\oplus\begin{smallmatrix} 2\end{smallmatrix}\oplus\begin{smallmatrix} 3\end{smallmatrix}\oplus\begin{smallmatrix} 5\end{smallmatrix}\oplus\begin{smallmatrix} 6\end{smallmatrix}\oplus\begin{smallmatrix} 2\\1\end{smallmatrix}\oplus\begin{smallmatrix} 3\\2\end{smallmatrix}\oplus\begin{smallmatrix} 4\\3\end{smallmatrix}\oplus\begin{smallmatrix} 6\\5\end{smallmatrix}\oplus\begin{smallmatrix} 7\\6\end{smallmatrix}\oplus\begin{smallmatrix} 5\\4\\3 \end{smallmatrix})$.

\end{example}

\section{Representation-finiteness of higher Auslander algebras}
In this section, we investigate the representation-finiteness of $n$-Auslander algebras in terms of modules with projective dimension $n+1$.

The following result connects projective covers and injective envelopes in short exact sequences.

\begin{prop}{\rm \cite[Lemma 2.6]{HZ}}
Let $0\rightarrow X\xrightarrow{f} Q\xrightarrow{g} Y\rightarrow 0$ be a non-split short exact sequence where $Q$ is a projective-injective module. Then the following statements are equivalent.

{\rm (1)} $Y$ is indecomposable and $g$ is a projective cover of $Y$.

{\rm (2)} $X$ is indecomposable and $f$ is an injective envelope of $X$.
\end{prop}

Now we prove the main result of this section.
\begin{thm}
 Let $\Lambda$ be an $n$-Auslander algebra. Then $\Lambda$ is representation-finite if and only if the number of non-isomorphic indecomposable $\Lambda$-modules with projective dimension $n+1$ is finite.
\end{thm}

\begin{proof}
If $\Lambda$ is representation-finite, there are finitely many non-isomorphic indecomposable modules with projective dimension $n+1$.

Now assume the number of non-isomorphic indecomposable $\Lambda$-modules with projective dimension $n+1$ is finite. Let $Q$ be a basic projective-injective module such that ${\add}\,Q={\prinj}(\Lambda)$. For an indecomposable module $X$ with projective dimension $i$ for some integer $1\leq i\leq n$, it follows from Proposition 2.3 that $X$ belongs to ${\sub}\,\Lambda={\sub}\,Q$. Then there exists a short exact sequence $0\rightarrow X\xrightarrow{f_{0}} Q_{0}\xrightarrow{g_{0}} X_{1}\rightarrow 0$ where $f_{0}$ is an injective envelope of $X$ and $Q_{0}\in{\add}\,Q$. According to Proposition 3.1, $X_{1}$ is indecomposable and $g_{0}$ is a projective cover of $X_{1}$. Note that the projective dimension of $X_{1}$ is $i+1$. This means that any indecomposable module with projective dimension $i$ is the syzygy of an indecomposable module with projective dimension $i+1$. Thus
if there are only finitely many indecomposable modules with projective dimension $i+1$, so are the indecomposable modules with projective dimension $i$. It follows from the assumption that the number of non-isomorphic indecomposable modules with projective dimension $i$ for $1\leq i \leq n$ is finite. This completes our proof.
\end{proof}

\begin{remark}\rm
For $X_{1}$ in the above proof, if ${\pd}\,X_{1}\leq n$, there exists a short exact sequence $0\rightarrow X_{1}\xrightarrow{f_{1}} Q_{1}\xrightarrow{g_{1}} X_{2}\rightarrow 0$ where $f_{1}$ is an injective envelope of $X_{1}$ and $Q_{1}\in{\add}\,Q$. Moreover, $X_{2}$ is indecomposable with ${\pd}\,X_{2}=i+2$ and $g_{1}$ is a projective cover of $X_{2}$. Continuing this procedure, we get a long exact sequence $$0\rightarrow X\rightarrow Q_{0}\rightarrow Q_{1}\rightarrow \cdots \rightarrow Q_{n-i}\rightarrow X_{n+1-i}\rightarrow 0$$ where $Q_{j}\in{\add}\,Q$ for $0\leq j\leq n-i$ and $X_{n+1-i}$ is indecomposable with ${\pd}\,X_{n+1-i}=n+1$. Then $X\cong\Omega^{n+1-i}X_{n+1-i}$ and any indecomposable non-projective module with projective dimension at most $n$ occurs as some syzygy in the minimal projective resolution of an indecomposable module with projective dimension $n+1$. Thus it is natural to consider whether the number of indecomposable modules with projective dimension $n+1$ can control the number of all indecomposable modules.
\end{remark}

We recover the following characterization of representation-finite Auslander algebras introduced by Zito in \cite{Z}.

\begin{corollary}{\rm (\cite[Theorem 1.2]{Z}) }
Let $\Lambda$ be an Auslander algebra. Then $\Lambda$ is representation-finite if and only if the number of non-isomorphic indecomposable modules with projective dimension $2$ is finite.
\end{corollary}

\begin{example}\rm
Let $\Lambda$ be an algebra given by the following quiver with relations $\alpha_{i}\alpha_{i+1}=0$ for $1\leq i\leq n$.

\qquad \qquad \qquad \qquad $\small{1\xleftarrow{\alpha_{1}}2\xleftarrow{\alpha_{2}}\cdots\xleftarrow{\alpha_{n}}n+1\xleftarrow{\alpha_{n+1}}n+2}$

Then $\Lambda$ is an $n$-Auslander algebra and the only indecomposable module with projective dimension $n+1$ is the simple module $S_{n+2}$. By Theorem 3.2, $\Lambda$ is of finite representation type. In fact, $\Lambda$ is representation-finite since it is a Nakayama algebra.
\end{example}

For a representation-finite $n$-Auslander algebra, we classify the modules with projective dimension $n+1$ by their injective dimensions and calculate the number of non-isomorphic indecomposable non-projective modules.

\begin{prop}
Let $\Lambda$ be a representation-finite $n$-Auslander algebra and $t_{i}$ be the number of non-isomorphic indecomposable $\Lambda$-modules with projective dimension $n+1$ and injective dimension $i$ for $0\leq i\leq n+1$. Then the number of non-isomorphic indecomposable non-projective $\Lambda$-modules is $\sum^{n+1}_{j=1}(n-j+2)t_{j}+(n+1)t_{0}$.
\end{prop}

\begin{proof}
According to the dual of Proposition 2.3, a $\Lambda$-module $X$ has injective dimension at most $n$ if and only if it is generated by an injective module, or equivalently, by a projective-injective module. Then by Remark 3.3, any indecomposable non-projective module with projective dimension at most $n$ occurs in the minimal projective resolution of an indecomposable module with projective dimension $n+1$ and injective dimension at most $n$.

Let $Y$ be an indecomposable $\Lambda$-module with ${\pd}\,Y=n+1$ and ${\id}\,Y=i$ for some integer $1\leq i\leq n$. Consider the following minimal projective resolution of $Y$.
$$0\rightarrow P_{n+1}\rightarrow \cdots \rightarrow P_{n-i+1}\rightarrow P_{n-i}\rightarrow \cdots\rightarrow P_{0}\rightarrow Y\rightarrow 0$$
Since ${\id}\,Y\leq n$, $P_{0}$ is projective-injective. It follows from Proposition 3.1 that $\Omega Y$ is indecomposable and the monomorphism $\Omega Y\rightarrow P_{0}$ is an injective envelope of $\Omega Y$. Moreover, we have ${\id}\,\Omega Y=i+1$. Continuing this procedure, we get

(1) $P_{j}$ is projective-injective for $0\leq j\leq n-i$ and $P_{n-i+1}$ is not injective.

(2) $\Omega^{j}Y$ is indecomposable and ${\id}\,\Omega^{j}Y=i+j$ for $1\leq j\leq n-i+1$.

(3) The monomorphism $\Omega^{j}Y\rightarrow P_{j-1}$ is an injective envelope of $\Omega^{j}Y$ for $1\leq j\leq n-i+1$.

Following Remark 3.3, the number of indecomposable non-projective modules occurring in the minimal projective resolution of $Y$ is $n-i+1$.

Let $I$ be an indecomposable injective $\Lambda$-module with ${\pd}\,I=n+1$. Consider the following minimal projective resolution of $I$.
$$0\rightarrow P^{'}_{n+1}\rightarrow  P^{'}_{n}\rightarrow\cdots\rightarrow P^{'}_{1}\rightarrow P^{'}_{0}\rightarrow I\rightarrow 0$$
We know that $P^{'}_{j}$ is projective-injective for $0\leq j\leq n$ and $P^{'}_{n+1}$ is not injective. Since $\Omega^{n+1}I=P^{'}_{n+1}$ is projective, the number of indecomposable non-projective modules occurring in the minimal projective resolution of $I$ is $n$.

Thus for the representation-finite $n$-Auslander algebra $\Lambda$, the number of non-isomorphic indecomposable non-projective modules is

$\sum^{n}_{j=1}(n-j+1)t_{j}+nt_{0}+\sum^{n+1}_{j=0}t_{j}=\sum^{n+1}_{j=1}(n-j+2)t_{j}+(n+1)t_{0}.$
\end{proof}

In Example 3.5, we have $t_{0}=1$ and $t_{j}=0$ for $1\leq j\leq n+1$. Thus the number of non-isomorphic indecomposable non-projective modules is $n+1$.

\section{Higher Auslander algebras of linearly oriented type $\mathbb{A}$}
In this section, we provide a classification of the representation-finite higher Auslander algebras of linearly oriented type $\mathbb{A}$. Moreover, we calculate the number of non-isomorphic indecomposable modules over these algebras.

\subsection{Quivers of higher Auslander algebras of linearly oriented type $\mathbb{A}$}
An algebra $\Lambda$ is called \emph{$n$-representation finite} if it admits an $n$-cluster tilting module $M$ and its global dimension ${\gl}\,\Lambda\leq n$. For an $n$-representation finite algebra $\Lambda$, the basic $n$-cluster tilting module $M$ is unique and the endomorphism algebra ${\End}_{\Lambda}M$ is called the \emph{$n$-Auslander algebra of $\Lambda$}. Higher Auslander algebras of linearly oriented type $\mathbb{A}$ introduced by Iyama in \cite{I3} are a remarkable family of $n$-representation finite algebras for $n\geq1$. These algebras are constructed recursively as follows.

Let $A_{m}$ be the linearly oriented quiver $\small{1\rightarrow 2\rightarrow \cdots \rightarrow m}$ and $A_{m}^{1}=kA_{m}$ be the path algebra of $A_{m}$. Then $A_{m}^{1}$ is a $1$-representation finite algebra and the unique basic $1$-cluster tilting module $_{A_{m}^{1}}M$ is the additive generator of the category $A_{m}^{1}\text{-}{\mo}$. Denote by $A_{m}^{2}={\End}_{A_{m}^{1}}M$ the Auslander algebra of $A_{m}^{1}$. Then $A_{m}^{2}$ is a $2$-representation finite algebra and there exists a unique basic $2$-cluster tilting module $_{A_{m}^{2}}M$. Continuing this procedure, it follows from \cite[Theorem 1.18]{I3} that the $(n-1)$-Auslander algebra $A_{m}^{n}={\End}_{A_{m}^{n-1}}M$ of $A_{m}^{n-1}$ is an $n$-representation finite algebra and there exists a unique basic $n$-cluster tilting module $_{A_{m}^{n}}M$. The algebra $A_{m}^{n}$ is called the \emph{$(n-1)$-Auslander algebra of linearly oriented type $A_{m}$}. For these algebras, Iyama gave explicit descriptions of their quivers with relations, see \cite[Section 6]{I3} for details. In this paper, we follow the descriptions of $A_{m}^{n}$ in \cite{HJ} with a slight modification.

Set $\mathcal{V}_{m}^{n}=\{(x_{1},x_{2},\cdots,x_{n})\in\mathbb{Z}^{n}\,|\,1\leq x_{1}<x_{2}<\cdots<x_{n}\leq m+n-1\}$. For $x\in\mathcal{V}_{m}^{n}$ and $i\in\{1,2,\cdots,n\}$, denote $$x+e_{i}=(x_{1},\cdots,x_{i-1},x_{i}+1,x_{i+1},\cdots,x_{n})$$ if $x+e_{i}$ is still in $\mathcal{V}_{m}^{n}$ where $e_{1},e_{2},\cdots,e_{n}$ is the standard basis of $\mathbb{Z}^{n}$.

Now define a quiver $Q_{m}^{n}$ with vertices $\mathcal{V}_{m}^{n}$ and arrows $a_{i}(x):x\rightarrow x+e_{i}$ for $x\in\mathcal{V}_{m}^{n}$ and $i\in\{1,2,\cdots,n\}$. For $i\neq j$ and $x\in\mathcal{V}_{m}^{n}$, define a relation
$$\rho_{ij}^{x}=a_{j}(x+e_{i})a_{i}(x)-a_{i}(x+e_{j})a_{j}(x).$$
When $x+e_{i}$ or $x+e_{j}$ is not in $\mathcal{V}_{m}^{n}$, we have a zero relation $\rho_{ij}^{x}=a_{i}(x+e_{j})a_{j}(x)$ or $\rho_{ij}^{x}=a_{j}(x+e_{i})a_{i}(x)$.

Let $I_{m}^{n}$ be the ideal of the path algebra $kQ_{m}^{n}$ generated by all $\rho_{ij}^{x}$. Then the Auslander-Reiten quiver of ${\add}_{A_{m}^{n-1}}M$(see \cite[Definition 6.1]{I3}) is $Q_{m}^{n}$ and $A_{m}^{n}=kQ_{m}^{n}/I_{m}^{n}$.
Here it is different from that in \cite{HJ}. Because we consider left modules and use the functor ${\Hom}_{A_{m}^{n-1}}(-,_{A_{m}^{n-1}}M):A_{m}^{n-1}\text{-}{\mo}\rightarrow A_{m}^{n}\text{-}{\mo}$, the quiver of $A_{m}^{n}={\End}_{A_{m}^{n-1}}M$ is the same as the Auslander-Reiten quiver of ${\add}_{A_{m}^{n-1}}M$.

\begin{example}\rm
(1) The quiver $Q_{2}^{4}$ of the algebra $A_{2}^{4}$ is as follows.
{\small$${1234\rightarrow 1235\rightarrow 1245\rightarrow 1345\rightarrow 2345}$$}
(2) The quiver $Q_{3}^{3}$ of the algebra $A_{3}^{3}$ is as follows.
{\small$$\xymatrix{
&&145\ar[r]&245\ar[r]&345\\
&125\ar[r]&135\ar[r]\ar[u]&235\ar[u]&\\
123\ar[r]&124\ar[u]\ar[r]&134\ar[u]\ar[r]&234\ar[u]&
}$$}
(3) The quiver $Q_{3}^{4}$ of the algebra $A_{3}^{4}$ is as follows.
{\small$$\xymatrix{
&&&1456\ar[r]&2456\ar[r]&3456\\
&&1256\ar[r]&1356\ar[u]\ar[r]&2356\ar[u]&\\
&1236\ar[r]&1246\ar[u]\ar[r]&1346\ar[u]\ar[r]&2346\ar[u]&\\
1234\ar[r]&1235\ar[u]\ar[r]&1245\ar[u]\ar[r]&1345\ar[r]\ar[u]&2345\ar[u]&
}$$}

\end{example}
For the set $\mathcal{V}_{m}^{n}$, consider two subsets $$\mathcal{V}_{m}^{n'}=\{(x_{1},x_{2},\cdots,x_{n})\in\mathcal{V}_{m}^{n}\,|\, x_{n}=m+n-1\}\ \text{and}$$ $$\mathcal{V}_{m}^{n''}=\mathcal{V}_{m}^{n}\setminus \mathcal{V}_{m}^{n'}=\{(x_{1},x_{2},\cdots,x_{n})\in\mathcal{V}_{m}^{n}\,|\, x_{n}<m+n-1\}.$$ Denote by $Q_{m}^{n'}$ and $Q_{m}^{n''}$ the full subquivers of $Q_{m}^{n}$ consisting of vertices $\mathcal{V}_{m}^{n'}$ and $\mathcal{V}_{m}^{n''}$, respectively. It is easily checked that the map $$(x_{1},\cdots,x_{n-1})\rightarrow (x_{1},\cdots,x_{n-1}, m+n-1)$$ from $\mathcal{V}_{m}^{n-1}$ to $\mathcal{V}_{m}^{n}$ induces an isomorphism $Q_{m}^{n-1}\cong Q_{m}^{n'}$. Meanwhile, the isomorphism $Q_{m-1}^{n}\cong Q_{m}^{n''}$ is given by the inclusion $\mathcal{V}_{m-1}^{n}\subseteq \mathcal{V}_{m}^{n}$. Thus the quiver $Q_{m}^{n}$ is a disjoint union of $Q_{m-1}^{n}$ and $Q_{m}^{n-1}$ with additional arrows $(x_{1},\cdots,x_{n-1},m+n-2)\rightarrow (x_{1},\cdots,x_{n-1})$ where $(x_{1},\cdots,x_{n-1},m+n-2)$ runs through all vertices of this form in $\mathcal{V}_{m-1}^{n}$.

\begin{prop}
Let $_{A_{m}^{n}}Q$ be a basic projective-injective $A_{m}^{n}$-module such that ${\add}_{A_{m}^{n}}Q={\prinj}(A_{m}^{n})$. Then there exist two equivalences ${\fac}(\upsilon^{-1}_{A_{m}^{n}}Q)\cong A_{m}^{n-1}\text{-}{\mo}$ and $^{\perp}(_{A_{m}^{n}}Q)\cong A_{m-1}^{n}\text{-}{\mo}$.
\end{prop}
\begin{proof}
Since the global dimension ${\gl}\,A_{m}^{n-1}\leq n-1$, the first equivalence follows from Proposition 2.6 (2) and (3).

According to \cite[Theorem 2.16(2)]{HJ}, the vertices $\mathcal{V}_{m}^{n''}$ in $Q_{m}^{n}$ correspond to all the indecomposable non-injective direct summands of $_{A_{m}^{n-1}}M$. By Proposition 2.1, the simple $A_{n}^{m}$-modules associated to the vertices $\mathcal{V}_{m}^{n''}$ have projective dimension $n$. Let $S^{''}$ be the set of all the simple $A_{n}^{m}$-modules associated to the vertices $\mathcal{V}_{m}^{n''}$. Then we have

$^{\perp}(_{A_{m}^{n}}Q)=\{X\in A_{m}^{n}\text{-}{\mo}\,|\, {\pd}\,S=n \;\text{for any composition factor} \;S\; \text{of}\; X\}$

\qquad\quad\; $=\{X\in A_{m}^{n}\text{-}{\mo}\,|\, \text{Any composition factor of}\; X\; \text{belongs to} \; S^{''}\}.$

Since there is an isomorphism between the quiver $Q_{m-1}^{n}$ of $A_{m-1}^{n}$ and the subquiver $Q_{m}^{n''}$, we get the second equivalence  $^{\perp}(_{A_{m}^{n}}Q)\cong A_{m-1}^{n}\text{-}{\mo}$.
\end{proof}

\begin{remark} \rm
The above proof also shows that any simple $A_{n}^{m}$-modules associated to the vertices $\mathcal{V}_{m}^{n'}$ has projective dimension at most $n-1$.
\end{remark}

For an $A_{m}^{n}$-module $X$ with projective dimension $n$, by Proposition 2.3, its socle has projective dimension $n$. Then $X$ either belongs to $^{\perp}(_{A_{m}^{n}}Q)$ or has a composition factor with projective dimension at most $n-1$. In order to classify representation-finite $A_{m}^{n}$, we study the number of indecomposable modules in  $^{\perp}(_{A_{m}^{n}}Q)$ and those admitting composition factors with projective dimension at most $n-1$.

\subsection{Representation-finite higher Auslander algebras of linearly oriented type $\mathbb{A}$ }
Since $A_{1}^{n}$ is semisimple for $n\geq 1$ and $A_{m}^{1}$ is the path algebra of linearly oriented quiver $A_{m}$ for $m\geq 1$, they are representation-finite. Thus we only consider $A_{m}^{n}$ for $m,n\geq 2$. Denote by $|{\ind}\,A_{m}^{n}|$ the number of non-isomorphic indecomposable $A_{m}^{n}$-modules.

\begin{prop}
The $(n-1)$-Auslander algebra $A_{2}^{n}$ of linearly oriented type $A_{2}$ is representation-finite and $|{\ind}\,A_{2}^{n}|=2n+1$.
\end{prop}
\begin{proof}
It is easy to check that $A_{2}^{n}$ is the Nakayama algebra given by the quiver $\small{1\xleftarrow{\alpha_{1}}2\xleftarrow{\alpha_{2}}\cdots \xleftarrow{\alpha_{n-1}}n\xleftarrow{\alpha_{n}}n+1}$ with relations $\alpha_{i}\alpha_{i+1}=0$ for $1\leq i\leq n-1$. Then $A_{2}^{n}$ is representation-finite. According to Proposition 3.6, $t_{0}=1$ and the number of indecomposable non-projective modules is $nt_{0}=n$. Thus we have $|{\ind}\,A_{2}^{n}|=n+(n+1)=2n+1$.
\end{proof}

\begin{prop}
The $(n-1)$-Auslander algebra $A_{3}^{n}$ of linearly oriented type $A_{3}$ is representation-finite and $|{\ind}\,A_{3}^{n}|=\frac{(n+1)(n^{2}+5n+3)}{3}$.
\end{prop}

\begin{proof}
For simplicity, we denote by $\overline{pq}$ the vertex $(x_{1},x_{2},\cdots,x_{n})$ in $\mathcal{V}_{3}^{n}$ such that $x_{i}\neq p,q$ for $1\leq i\leq n$ if $1\leq p<q\leq n+2$. Then the quiver of $A_{3}^{n}$ is as follows.

\noindent
\begin{tikzpicture}[->][scale=.85]
\node(s1) at (0,0) {\tiny$\overline{(n+1)(n+2)}$};
\node(s2) at (2,0) {\tiny$\overline{n(n+2)}$};
\node(s3) at (4,0) {\tiny$\overline{(n-1)(n+2)}$};
\node(s4) at (5.7,0) {\tiny$\cdots$};
\node(s5) at (7,0) {\tiny$\overline{3(n+2)}$};
\node(s6) at (8.5,0) {\tiny$\overline{2(n+2)}$};
\node(s7) at (10,0) {\tiny$\overline{1(n+2)}$};

\node(t2) at (2,1) {\tiny$\overline{n(n+1)}$};
\node(t3) at (4,1) {\tiny$\overline{(n-1)(n+1)}$};
\node(t4) at (5.7,1) {\tiny$\cdots$};
\node(t5) at (7,1) {\tiny$\overline{3(n+1)}$};
\node(t6) at (8.5,1) {\tiny$\overline{2(n+1)}$};
\node(t7) at (10,1) {\tiny$\overline{1(n+1)}$};

\node(a3) at (4,2) {\tiny$\vdots$};
\node(a5) at (7,2) {\tiny$\vdots$};
\node(a6) at (8.5,2) {\tiny$\vdots$};
\node(a7) at (10,2) {\tiny$\vdots$};

\node(b5) at (7,3) {\tiny$\overline{34}$};
\node(b6) at (8.5,3) {\tiny$\overline{24}$};
\node(b7) at (10,3) {\tiny$\overline{14}$};

\node(c6) at (8.5,4) {\tiny$\overline{23}$};
\node(c7) at (10,4) {\tiny$\overline{13}$};
\node(c8) at (11,4){\tiny$\overline{12}$};

\draw (s1)--(s2);
\draw (s2)--(s3);
\draw (s3)--(s4);
\draw (s4)--(s5);
\draw (s5)--(s6);
\draw (s6)--(s7);

\draw (t2)--(t3);
\draw (t3)--(t4);
\draw (t4)--(t5);
\draw (t5)--(t6);
\draw (t6)--(t7);

\draw (s2)--(t2);
\draw (s3)--(t3);
\draw (s5)--(t5);
\draw (s6)--(t6);
\draw (s7)--(t7);

\draw (t3)--(a3);
\draw (t5)--(a5);
\draw (t6)--(a6);
\draw (t7)--(a7);

\draw (b5)--(b6);
\draw (b6)--(b7);
\draw (a5)--(b5);
\draw (a6)--(b6);
\draw (a7)--(b7);

\draw (c6)--(c7);
\draw (c7)--(c8);
\draw (b6)--(c6);
\draw (b7)--(c7);
\end{tikzpicture}

According to \cite[Theorem 2.16(2)]{HJ} and Proposition 2.1, the simple $A_{3}^{n}$-modules associated to the vertices {\small $\overline{i(n+2)}$} for $1\leq i\leq n+1$ are exactly all the simple modules with projective dimension $n$. By Proposition 4.2 and Proposition 4.4, there are only finitely many non-isomorphic indecomposable modules in $^{\perp}(_{A_{3}^{n}}Q)$. These indecomposable modules are injective modules
{\scriptsize $\overline{(n+1)(n+2)}$},{\small ${\overline{(n+1)(n+2)}\atop\overline{n(n+2)}}, {\overline{n(n+2)}\atop \overline{(n-1)(n+2)}},\cdots,{\overline{2(n+2)}\atop \overline{1(n+2)}}$} and simple modules {\small $S_{\overline{i(n+2)}}$} for $1\leq i\leq n$. The injective dimension of {\small $S_{\overline{i(n+2)}}$} is $n+1-i$ for $1\leq i\leq n$.

The indecomposable $A_{3}^{n}$-modules with projective dimension $n$ which do not belong to $^{\perp}(_{A_{3}^{n}}Q)$ are listed below:

\noindent {\small $\overline{n(n+2)}\atop {\overline{n(n+1)}\quad \overline{(n-1)(n-2)}}$} with injective dimension $1$;

\noindent {\small $\overline{(n-1)(n+2)}\atop{\overline{(n-1)(n+1)}\quad \overline{(n-2)(n+2)}}$,$\overline{n(n+1)}\quad\overline{(n-1)(n+2)}\atop{\quad\overline{(n-1)(n+1)}\quad \overline{(n-2)(n+2)}}$} with injective dimension $2$ and $1$, respectively;

\noindent {\small $\overline{(n-2)(n+2)}\atop{\overline{(n-2)(n+1)}\quad \overline{(n-3)(n+2)}}$,$\overline{(n-1)(n+1)}\quad\overline{(n-2)(n+2)}\atop{\quad\overline{(n-2)(n+1)}\quad \overline{(n-3)(n+2)}}$,\quad $\overline{(n-1)(n+1)}\quad\overline{(n-2)(n+2)}\atop{\overline{(n-1)n}\quad\overline{(n-2)(n+1)}\quad \overline{(n-3)(n+2)}}$}

\noindent with injective dimension $3,1$ and $2$, respectively;

\noindent {\small $\overline{(n-3)(n+2)}\atop{\overline{(n-3)(n+1)}\quad \overline{(n-4)(n+2)}}$,$\overline{(n-2)(n+1)}\quad\overline{(n-3)(n+2)}\atop{\quad\overline{(n-3)(n+1)}\quad \overline{(n-4)(n+2)}}$,\quad $\overline{(n-2)(n+1)}\quad\overline{(n-3)(n+2)}\atop{\overline{(n-2)n}\quad\overline{(n-3)(n+1)}\quad \overline{(n-4)(n+2)}}$,}

\noindent {\small $\overline{(n-1)n}\quad\overline{(n-2)(n+1)}\quad\overline{(n-3)(n+2)}\atop{\quad\overline{(n-2)n}\quad\overline{(n-3)(n+1)}\quad \overline{(n-4)(n+2)}}$}

\noindent with injective dimension $4,1,3$ and $2$, respectively;

\noindent $\cdots \cdots$

\noindent {\small $\overline{2(n+2)}\atop{\overline{2(n+1)}\quad \overline{1(n+2)}}$,$\overline{3(n+1)}\quad\overline{2(n+2)}\atop{\quad\overline{2(n+1)}\quad \overline{1(n+2)}}$,\quad $\overline{3(n+1)}\quad\overline{2(n+2)}\atop{\overline{3n}\quad\overline{2(n+1)}\quad \overline{1(n+2)}}$,$\overline{4n}\quad\overline{3(n+1)}\quad\overline{2(n+2)}\atop{\quad\overline{3n}\quad\overline{2(n+1)}\quad \overline{1(n+2)}}$},

\noindent {\small $\overline{4n}\quad\overline{3(n+1)}\quad\overline{2(n+2)}\atop{\overline{4(n-1)}\quad\overline{3n}\quad\overline{2(n+1)}\quad \overline{1(n+2)}}$},\quad$\cdots$,\quad{\small${\quad\overline{(\frac{n}{2}+1)(\frac{n}{2}+3)}\;\cdots\;\overline{4n}\quad\overline{3(n+1)}\quad\overline{2(n+2)}}\atop{\overline{(\frac{n}{2}+1)(\frac{n}{2}+2)}\quad \overline{\frac{n}{2}(\frac{n}{2}+3)}\;\cdots\;\overline{3n}\quad\overline{2(n+1)}\quad\overline{1(n+2)}}$}

\noindent with injective dimension $n-1,1,n-2,2,n-3,\cdots, \frac{n}{2}+1,\frac{n}{2}-1$ and $ \frac{n}{2}$, respectively if $n$ is even or

\noindent {\small $\overline{2(n+2)}\atop{\overline{2(n+1)}\quad \overline{1(n+2)}}$,$\overline{3(n+1)}\quad\overline{2(n+2)}\atop{\quad\overline{2(n+1)}\quad \overline{1(n+2)}}$,\quad $\overline{3(n+1)}\quad\overline{2(n+2)}\atop{\overline{3n}\quad\overline{2(n+1)}\quad \overline{1(n+2)}}$,$\overline{4n}\quad\overline{3(n+1)}\quad\overline{2(n+2)}\atop{\quad\overline{3n}\quad\overline{2(n+1)}\quad \overline{1(n+2)}}$},

\noindent {\small $\overline{4n}\quad\overline{3(n+1)}\quad\overline{2(n+2)}\atop{\overline{4(n-1)}\quad\overline{3n}\quad\overline{2(n+1)}\quad \overline{1(n+2)}}$},\quad$\cdots$,\quad{\small${\overline{(\frac{n+1}{2}+1)(\frac{n+1}{2}+2)}\;\cdots\;\overline{4n}\quad\overline{3(n+1)}\quad\overline{2(n+2)}}\atop{\qquad\overline{\frac{n+1}{2}(\frac{n+1}{2}+2)} \;\cdots\;\overline{3n}\quad\overline{2(n+1)}\quad\overline{1(n+2)}}$}

\noindent with injective dimension $n-1,1,n-2,2,n-3,\cdots,\frac{n+1}{2}$ and $\frac{n-1}{2}$, respectively if $n$ is odd.

Then there are {\small$$(2n+1)+(1+2+3+\cdots+n-1)=\frac{(n+1)(n+2)}{2}$$} non-isomorphic indecomposable modules with projective dimension $n$. Thus $A_{3}^{n}$ is representation-finite by Theorem 3.2.

For $0\leq i\leq n$, it is easy to check that the number $t_{i}$ of non-isomorphic indecomposable modules with projective dimension $n$ and injective dimension $i$ is $n+1-i$. Moreover, there are {\small$$1+2+3+\cdots+n+(n+1)=\frac{(n+1)(n+2)}{2}$$}
\!non-isomorphic indecomposable projective modules. By Proposition 3.6, we have

$|{\ind}\,A_{3}^{n}|={\small\sum^{n}_{j=1}(n-j+1)t_{j}+nt_{0}+\frac{(n+1)(n+2)}{2}=\frac{(n+1)(n^{2}+5n+3)}{3}}.$

\end{proof}

\begin{example}\rm
(1) The quiver of $A_{3}^{4}$ is as follows.
{\small 
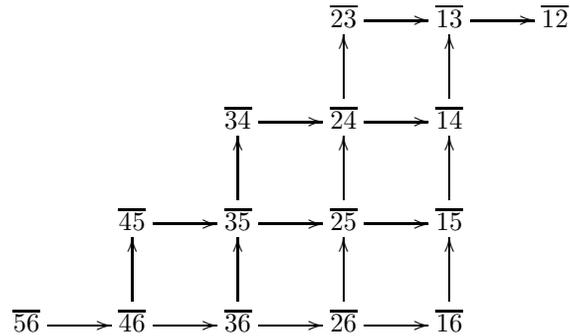
$${\xymatrix{
&&&\overline{23}\ar[r]&\overline{13}\ar[r]&\overline{12}\\
&&\overline{34}\ar[r]&\overline{24}\ar[u]\ar[r]&\overline{14}\ar[u]&\\
&\overline{45}\ar[r]&\overline{35}\ar[u]\ar[r]&\overline{25}\ar[u]\ar[r]&\overline{15}\ar[u]&\\
\overline{56}\ar[r]&\overline{46}\ar[u]\ar[r]&\overline{36}\ar[u]\ar[r]&\overline{26}\ar[r]\ar[u]&\overline{16}\ar[u]&
}}$$}

The indecomposable modules with projective dimension $4$ which do not belong to $^{\perp}(_{A_{3}^{4}}Q)$ are \;$\begin{smallmatrix}&\overline{46}&\\\overline{45}&&\overline{36}\end{smallmatrix}$, \;$\begin{smallmatrix}&\overline{36}&\\\overline{35}&&\overline{26}\end{smallmatrix}$, \;$\begin{smallmatrix}\overline{45}&&\overline{36}&\\&\overline{35}&&\overline{26}\end{smallmatrix}$,
\;$\begin{smallmatrix}&\overline{26}&\\\overline{25}&&\overline{16}\end{smallmatrix}$,
$\begin{smallmatrix}\overline{35}&&\overline{26}&\\&\overline{25}&&\overline{16}\end{smallmatrix}$

\noindent and $\begin{smallmatrix}&\overline{35}&&\overline{26}&\\\overline{34}&&\overline{25}&&\overline{16}\end{smallmatrix}$.

There are 65 non-isomorphic indecomposable $A_{3}^{4}$-modules.

(2) The quiver of $A_{3}^{5}$ is as follows.
{\small 
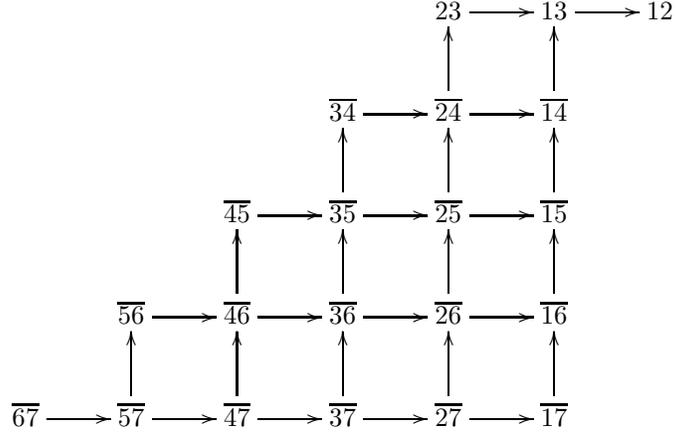
$$\xymatrix{
&&&&\overline{23}\ar[r]&\overline{13}\ar[r]&\overline{12}\\
&&&\overline{34}\ar[r]&\overline{24}\ar[u]\ar[r]&\overline{14}\ar[u]&\\
&&\overline{45}\ar[r]&\overline{35}\ar[u]\ar[r]&\overline{25}\ar[u]\ar[r]&\overline{15}\ar[u]&\\
&\overline{56}\ar[r]&\overline{46}\ar[u]\ar[r]&\overline{36}\ar[u]\ar[r]&\overline{26}\ar[r]\ar[u]&\overline{16}\ar[u]&\\
\overline{67}\ar[r]&\overline{57}\ar[r]\ar[u]&\overline{47}\ar[r]\ar[u]&\overline{37}\ar[r]\ar[u]&\overline{27}\ar[r]\ar[u]&\overline{17}\ar[u]
}$$}

The indecomposable modules with projective dimension $5$ which do not belong to $^{\perp}(_{A_{3}^{5}}Q)$ are \;$\begin{smallmatrix}&\overline{57}&\\\overline{56}&&\overline{47}\end{smallmatrix}$, \;$\begin{smallmatrix}&\overline{47}&\\\overline{46}&&\overline{37}\end{smallmatrix}$, \;$\begin{smallmatrix}\overline{56}&&\overline{47}&\\&\overline{46}&&\overline{37}\end{smallmatrix}$,
\;$\begin{smallmatrix}&\overline{37}&\\\overline{36}&&\overline{27}\end{smallmatrix}$,
\;$\begin{smallmatrix}\overline{46}&&\overline{37}&\\&\overline{36}&&\overline{27}\end{smallmatrix}$,

\noindent \;$\begin{smallmatrix}&\overline{46}&&\overline{37}&\\\overline{45}&&\overline{36}&&\overline{27}\end{smallmatrix}$,
\;$\begin{smallmatrix}&\overline{27}&\\\overline{26}&&\overline{17}\end{smallmatrix}$,
\;$\begin{smallmatrix}\overline{36}&&\overline{27}&\\&\overline{26}&&\overline{17}\end{smallmatrix}$,
\;$\begin{smallmatrix}&\overline{36}&&\overline{17}&\\\overline{35}&&\overline{26}&&\overline{17}\end{smallmatrix}$ \, and
\;$\begin{smallmatrix}\overline{45}&&\overline{36}&&\overline{27}&\\&\overline{35}&&\overline{26}&&\overline{17}\end{smallmatrix}$.

There are 106 non-isomorphic indecomposable $A_{3}^{5}$-modules.
\end{example}

The following result is useful in the proof of our main result in this section.

\begin{prop}{\rm \cite[Theorem 3.1(i)]{M}}
Let $T_{2}(A_{m}^{1})=\left( \begin{smallmatrix}
                       A_{m}^{1} & 0 \\
                       A_{m}^{1} & A_{m}^{1}
                      \end{smallmatrix}
                      \right) $
be the triangular matrix algebra of $A_{m}^{1}$. Then $T_{2}(A_{m}^{1})$ is representation-finite if $m\leq 4$, tame infinite if $m=5$ and wild otherwise.
\end{prop}

Now we give a classification of representation-finite $A_{m}^{n}$.

\begin{thm}
Let $A_{m}^{n}$ be the $(n-1)$-Auslander algebra of linearly oriented type $A_{m}$$(n,m\geq2)$. Then $A_{m}^{n}$ is representation-finite if and only if one of the followings holds:

 {\rm(1)} $m=2$. In this case, $|{\ind}\,A_{2}^{n}|=2n+1$.

 {\rm(2)} $m=3$. In this case, $|{\ind}\,A_{3}^{n}|=\frac{(n+1)(n^{2}+5n+3)}{3}$.

 {\rm(3)} $n=2$ and $m=4$. In this case, $|{\ind}\,A_{4}^{2}|=56$.
\end{thm}

\begin{proof}
By the construction of $A_{m}^{n}$ and Proposition 2.6(2), we know that if $A_{m}^{n}$ is representation-infinite, so is $A_{m}^{n+1}$. Note that $A_{m}^{2}$ is the Auslander algebra of the path algebra $A_{m}^{1}=kA_{m}$. According to \cite[Theorem 1.1]{AR}, $A_{m}^{2}$ is representation-finite if and only if $T_{2}(A_{m}^{1})=\left( \begin{smallmatrix}
                       A_{m}^{1} & 0 \\
                       A_{m}^{1} & A_{m}^{1}
                      \end{smallmatrix}
                      \right) $
is representation-finite. Then it follows from Proposition 4.7 that $A_{m}^{2}$ is representation-infinite for $m\geq 5$. Thus $A_{m}^{n}$ is representation-infinite for $m\geq 5$ and $n\geq 2$.

The cases $m=2$ and $m=3$ follow from Proposition 4.4 and 4.5.

Assume $m=4$. The algebra $A_{4}^{2}$ is the Auslander algebra of the path algebra $A_{4}^{1}=k(1\rightarrow 2\rightarrow 3\rightarrow 4)$. It is easy to check that there are $56$ non-isomorphic indecomposable $A_{4}^{2}$-modules. The quiver of $A_{4}^{3}$ is as follows.

{\small $$\xymatrix{
          &156\ar[r]         &256\ar[r]                &356\ar[r]         &456        \\
          &146\ar[r]\ar[u]   &246\ar[r]\ar[u]          &346\ar[u]         &           \\
126\ar[r] &136\ar[r]\ar[u]   &236\ar[u]                &                  &           \\
          &                  &145\ar[luu]\ar[r]        &245\ar[r]\ar[luu] &345\ar[luu] \\
          &125\ar[r]\ar[luu] &135\ar[r]\ar[luu]\ar[u]  &235\ar[u]\ar[luu] &            \\
123\ar[r] &124\ar[r]\ar[u]   &134\ar[r]\ar[u]          &234\ar[u]         &
}$$}

This quiver contains the following subquiver of the Euclidean type $\widetilde{\mathbb{A}}_{5}$
{\small 
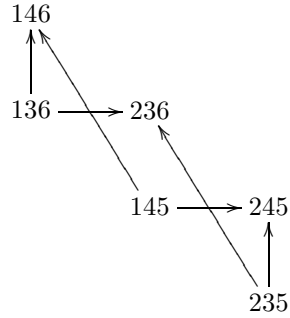
$$\xymatrix{
146&&\\
136\ar[r]\ar[u]&236&\\
&145\ar[r]\ar[luu]&245\\
&&235\ar[u]\ar[luu]
}$$}whose path algebra admits infinitely many non-isomorphic indecomposable modules. Then $A_{4}^{3}$ is representation-infinite and it implies that $A_{4}^{n}$ is of infinite representation type for $n\geq3$. This completes the proof.

\end{proof}

Shen Li: School of Science, Shandong Jianzhu University, PR China

\emph{E-mail address}: fbljs603@163.com


\small
\begin{thebibliography}{100}

\bibitem{A}\, M. Auslander, Representation Dimension of Artin Algebras, Lecture Notes, Queen Mary College, London, 1971.

\bibitem{AR}\, M. Auslander, I. Reiten, On the representation type of triangular matrix rings, J. London Math. Soc. 12(2)(1976) 371-382.

\bibitem{ASS}\, I. Assem, D. Simson, A. Skowronski, Elements of the representation theory of associative algebras. Vol.1. Techniques of representation theory,  London Mathematical Society Student Texts, 65. Cambridge University Press, Cambridge, 2006.

\bibitem{DJW}\, T. Dyckerhoff, G. Jasso, T. Walde, Simplicial structures in higher Auslander-Reiten theory, Adv. Math. 355(2019) 106762.

\bibitem{G}\, R. Gentle, T.T.F theories in abelian categories, Comm. Algebra 16(1996) 877-908.

\bibitem{HJ}\, M. Herschend, P. J${\o}$rgensen, Classification of higher wide subcategories for higher Auslander algebras of type $\mathbb{A}$, J. Pure Appl. Algebra 225(2021) 106583.

\bibitem{HZ}\, Z. Huang, X. Zhang, Higher Auslander algebras admitting trivial maximal orthogonal subcategories, J. Algebra 330(2011) 375-387.

\bibitem{I1}\, O. Iyama, Auslander correspondence, Adv. Math. 210(1)(2007) 51-82.

\bibitem{I2}\, O. Iyama, Higher-dimensional Auslander-Reiten theory on maximal orthogonal subcategories, Adv. Math. 210(2007) 22-50.

\bibitem{I4}\, O. Iyama, Auslander-Reiten theory revisited, Trends in Representation Theory of Algebras and Related Topics, European Mathematical Society Series of Congress Report, European Mathematical Society, Z\"{u}rich (2008) 349-397.

\bibitem{I3}\, O. Iyama, Cluster tilting for higher Auslander algebras, Adv. Math. 226(2011) 1-61.

\bibitem{IPTZ}\, K. Igusa, M. I. Platzeck, G. Todorov, D. Zacharia, Auslander algebras of finite representation type, Comm. Algebra 15(1-2)(1987) 377-424.

\bibitem{J}\, J. P. Jans, Some aspects of torsion, Pacific J. Math. 15 (1965) 1249-1259.

\bibitem{JK}\, G. Jasso, J. K\"{u}lshammer, Higher Nakayama algebras I: Construction, Adv. Math. 351(2019) 1139-1200. With an appendix by J. K\"{u}lshammer and C. Psaroudakis and an appendix by S. Kvamme.

\bibitem{KZ}\, F. Kong, P. Zhang, From CM-finite to CM-free, J. Pure Appl. Algebra 220(2016) 782-801.

\bibitem{LMZ}\, S. Li, R. Marczinzik, S. Zhang, Gorenstein projective dimensions of modules over minimal Auslander-Gorenstein algebras, Algebr Colloq 28(2)(2021) 337-350.

\bibitem{M}\, N. Marmaridis, On the representation type of certain triangular matrix algebras, Comm. Algebra 11(17)(1983) 1945-1964.

\bibitem{OT}\, S. Oppermann, H. Thomas, Higher-dimensional cluster combinatorics and representation theory, J. Eur. Math. Soc.(JEMS) 14(6)(2012) 1679-1737.

\bibitem{Z}\, S. Zito, Three results concerning Auslander algebras, Comm. Algebra 49(12)(2021) 5129-5136.

\end{thebibliography}
\end{document}